\documentclass{article}[11pt]

\usepackage{amsmath,amssymb,amsthm}
\usepackage[affil-it]{authblk}
\usepackage{enumitem}

\makeindex

\newtheorem{theorem}{Theorem}[section]
\newtheorem{lemma}[theorem]{Lemma}

\theoremstyle{definition}
\newtheorem{definition}[theorem]{Definition}
\newtheorem{remark}[theorem]{Remark}
\newtheorem{observation}[theorem]{Observation}

\author{Krishnendu Paul, Shameek Paul%
\thanks{E-mail addresses: \texttt{krishnendupaul@ggdcgopi2.ac.in, shameek.paul@rkmvu.ac.in}}}
\affil{\small
Government General Degree College Gopiballavpur-II,
P.O. Beliaberah, Dist. Jhargram, 721517, India \\

\bigskip
Ramakrishna Mission Vivekananda Educational and Research Institute, P.O. Belur Math, Dist. Howrah, 711202, India}

\date{}
\title {Doubly-weighted zero-sum constants}

\begin{document}

\baselineskip=14.5pt

\maketitle

\begin{abstract}
Let $A,B\subseteq\mathbb Z_n$ be given and $S=(x_1,\ldots, x_k)$ be a sequence in $\mathbb Z_n$. We say that $S$ is an $(A,B)$-weighted zero-sum sequence if there exist $a_1,\ldots,a_k\in A$ and $b_1,\ldots,b_k\in B$ such that $a_1x_1+\cdots+a_kx_k=0$ and $b_1a_1+\cdots+b_ka_k=0$. We show that if $S$ has length $2n-1$, then $S$ has an $(A,B)$-weighted zero-sum subsequence of length $n$. The constant $E_{A,B}$ is defined to be the smallest positive integer $k$ such that every sequence of length $k$ in $\mathbb Z_n$ has an $(A,B)$-weighted zero-sum subsequence of length $n$. A sequence in $\mathbb Z_n$ of length $E_{A,B}-1$ which does not have any $(A,B)$-weighted zero-sum subsequence of length $n$ is called an $E$-extremal sequence for $(A,B)$. We determine the constant $E_{A,B}$ and characterize the $E$-extremal sequences for some pairs $(A,B)$. We also study the related constants $C_{A,B}$ and $D_{A,B}$ which are defined in the article.
\end{abstract}

{\small \noindent
}

\section{Introduction}

By a sequence $S$ in a set $X$ of length $k$, we mean an element of the set $X^k$. Let $R$ be a non-zero ring with unity, let $A$ and $B$ be non-empty subsets of $R$, and let $M$ be an $R$-module. A sequence $(x_1,\ldots, x_k)$ in $M$ is called an {\it $A$-weighted zero-sum sequence} if there exist $a_1,\ldots,a_k\in A$ such that $a_1x_1+\cdots+a_kx_k=0$.

A sequence $(x_1,\ldots,x_k)$ in $M$ is called an {\it $(A,B)$-weighted zero-sum sequence} if there exist $a_1,\ldots,a_k\in A$ and $b_1,\ldots,b_k\in B$ such that $a_1x_1+\cdots+a_kx_k=0$ and $b_1a_1+\cdots +b_ka_k=0$. In particular, an $(A,B)$-weighted zero-sum sequence is also an $A$-weighted zero-sum sequence.

For $a,b\in\mathbb Z$ we denote the set $\{x\in\mathbb Z:a\leq x\leq b\}$ by $[a,b]$. We let $|A|$ denote the number of elements in a finite set $A$. We denote the subsets $\{0\}$ and $\{1\}$ of the ring $R$ by the boldface symbols $\mathbf 0$ and $\mathbf 1$ respectively. A sequence which is a $\mathbf 1$-weighted zero-sum sequence is simply called a {\it zero-sum sequence}. We now define the constants which we study in this article.

Let $(A,B)$ be a weight-set pair. The constant $C_{A,B}(M)$ is the least positive integer $k$ such that every sequence in $M$ of length $k$ has an $(A,B)$-weighted zero-sum subsequence having consecutive terms. The constant $D_{A,B}(M)$ is the least positive integer $k$ such that every sequence in $M$ of length $k$ has an $(A,B)$-weighted zero-sum subsequence.

The constant $E_{A,B}(M)$ is the least positive integer $k$ such that every sequence in $M$ of length $k$ has an $(A,B)$-weighted zero-sum subsequence of length $|M|$. The constants $C_A(M)$, $D_A(M)$, and $E_A(M)$ are defined to be the constants $C_{A,\mathbf 0}(M)$, $D_{A,\mathbf 0}(M)$, and $E_{A,\mathbf 0}(M)$ respectively. When $G$ is a finite abelian group and $A\subseteq \mathbb Z$, a related constant $s_A(G)$ has been studied in \cite{AGS}.

\begin{observation}\label{cca}
We see that $D_{A,B}(M)\leq C_{A,B}(M)$, $D_{A,B}(M)\leq E_{A,B}(M)$, and that $C_A(M)\leq C_{A,B}(M)$, $D_A(M)\leq D_{A,B}(M)$, and $E_A(M)\leq E_{A,B}(M)$.
\end{observation}

Let char $R$ be the characteristic of the ring $R$.

\begin{observation}\label{char}
Let $M$ be an $R$-module, let $A,B$ be non-empty subsets of $R$, and let $S=(x_1,\ldots,x_k)$ be a zero-sum sequence in $M$. Suppose char $R$ is positive and $k$ is a multiple of char $R$. We claim that $S$ is an $(A,B)$-weighted zero-sum sequence. Fix $a\in A$ and $b\in B$. As $S$ is a zero-sum sequence, we see that $x_1+\cdots+x_k=0$ and hence $ax_1+\cdots+ax_k=0$. Also, as $k$ is a multiple of char $R$, we see that $\underbrace{ba+\cdots+ba}_\text{$k$ times}=kba=0$. Hence, our claim is true.
\end{observation}

We give some conditions under which the constants $C_{A,B}(M)$, $D_{A,B}(M)$, and $E_{A,B}(M)$ exist.

\begin{theorem}\label{upbd}
Let $M$ be a finite $R$-module, let $A$ and $B$ be non-empty subsets of $R$, and let $m=|M|$. Suppose char $R$ is positive and $m$ is a multiple of char $R$. Then we have $C_{A,B}(M)\leq m^2$.
\end{theorem}

\begin{proof}
Let $S=(x_1,\ldots,x_k)$ be a sequence in $M$ of length $k=m^2$. For every $i\in [1,m]$ we let $y_i=x_1+x_2+\cdots+x_{im}\in M$. If all the $y_i$'s are distinct, then there exists $j\in [1,m]$ such that $y_j=0$. If not, there exist $i,j\in [1,m]$ with $i<j$ such that $y_i=y_j$ and so $x_{im+1}+\cdots +x_{jm}=y_j-y_i=0$. Thus, in both the cases, we get a zero-sum subsequence $T$ of $S$ having consecutive terms whose length is a multiple of $m$. As $m$ is a multiple of char $R$, by Observation \ref{char} we see that $T$ is an $(A,B)$-weighted zero-sum subsequence. Thus, it follows that $C_{A,B}(M)\leq m^2$.
\end{proof}

\begin{remark}\label{upbdd}
Let $M$ be a finite $R$-module and let $m=|M|$. By the Erd\H{o}s-Ginzburg-Ziv theorem \cite{EGZ}, we see that $E_\mathbf 1 (M)\leq 2m-1$. Let $A$ and $B$ be non-empty subsets of $R$. Suppose char $R$ is positive and $m$ is a multiple of char $R$. Since $E_\mathbf 1(M)\leq 2m-1$, by Observation \ref{char} it follows that $E_{A,B}(M)\leq 2m-1$.
\end{remark}

A sequence in $M$ of length $C_{A,B}(M)-1$ not having any $(A,B)$-weighted zero-sum subsequence of consecutive terms is called a {\it $C$-extremal sequence} for $(A,B)$. A sequence in $M$ of length $D_{A,B}(M)-1$ not having any $(A,B)$-weighted zero-sum subsequence is called a {\it $D$-extremal sequence} for $(A,B)$. A sequence in $M$ of length $E_{A,B}(M)-1$ not having any $(A,B)$-weighted zero-sum subsequence of length $|M|$ is called an {\it $E$-extremal sequence} for $(A,B)$.

Let $A$ be a non-empty subset of $R$. A sequence in $M$ of length $D_A(M)-1$ not having any $A$-weighted zero-sum subsequence is called a {\it $D$-extremal sequence} for $A$. We can also define $C$-extremal sequences for $A$ and $E$-extremal sequences for $A$ as in \cite{SKS2} and \cite{SKS3}.

Let $S=(x_1,\ldots,x_k)$ and $T=(y_1,\ldots,y_k)$ be sequences in $M$. We say that $S$ and $T$ are {\it equivalent} if there exists a permutation $\sigma$ of the set $[1,k]$ and there exists a unit $u\in R$ such that for every $i\in [1,k]$ we have $y_{\sigma(i)}=u\,x_i$. We say that $S$ and $T$ are {\it order-equivalent} if there exists a unit $u\in R$ such that for every $i\in [1,k]$ we have $y_i=u\,x_i$.

\begin{remark}\label{eq}
If $S$ is a $C$-extremal sequence for $(A,B)$ and if $T$ is order-equivalent to $S$, then $T$ is a $C$-extremal sequence for $(A,B)$. If $S$ is a $D$-extremal sequence for $(A,B)$ and if $T$ is equivalent to $S$, then $T$ is a $D$-extremal sequence for $(A,B)$. If $S$ is an $E$-extremal sequence for $(A,B)$ and if $T$ is equivalent to $S$, then $T$ is an $E$-extremal sequence for $(A,B)$.
\end{remark}

Let $S=(x_1,\ldots,x_k)$ be a sequence in a module $M$ and $x\in M$. We let $S+x$ be the sequence $(x_1+x,\ldots,x_k+x)$. We say that $S+x$ is a {\it translate} of $S$. We will use the following observation in Section \ref{sz1}.

\begin{observation}\label{tr}
Let $S$ be an $(A,\mathbf 1)$-weighted zero-sum sequence in a module $M$. Since for every $a_1,\ldots,a_k\in R$ and for every $x,x_1,\ldots,x_k\in M$ we have
\[a_1(x_1+x)+\cdots+a_k(x_k+x)=a_1x_1+\cdots +a_kx_k+(a_1+\cdots+a_k)x,\]
we see that every translate of $S$ is also an $(A,\mathbf 1)$-weighted zero-sum sequence.
\end{observation}

For every $n\in \mathbb N$ with $n\geq 2$ we denote the ring $\mathbb Z/n\mathbb Z$ by $\mathbb Z_n$ and we denote the set $\mathbb Z_n\setminus\{0\}$ by $\mathbb Z_n'$. We let $U(n)$ denote the group of units of $\mathbb Z_n$.  Henceforth, we consider the ring $\mathbb Z_n$ as a module over itself and we let $A$ and $B$ be non-empty subsets of $\mathbb Z_n'$. We denote the constant $E_{A,B}(\mathbb Z_n)$ simply as $E_{A,B}(n)$ or even as $E_{A,B}$ when the value of $n$ is clear from the weight-set pair $(A,B)$ and we denote the constant $E_A(\mathbb Z_n)$ by $E_A(n)$. The constant $E_A(n)$ has been studied extensively, for example in \cite{ACFKP}.

We adopt a similar notation for the constants $C_{A,B}(\mathbb Z_n)$ and $D_{A,B}(\mathbb Z_n)$. From Theorem \ref{upbd} and Remark \ref{upbdd} it follows that for every $n\in \mathbb N$ we have
\begin{equation}\label{ubzn}
C_{A,B}(n)\leq n^2 ~\mbox{ and }~ D_{A,B}(n)\leq E_{A,B}(n)\leq 2n-1.
\end{equation}
From the results in the next section, we see that the upper bounds in \eqref{ubzn} are sharp.

It is shown in \cite{Gr} that if $S$ is a sequence of $m+n-1$ elements from a finite abelian group $G$ of order $m$ and exponent $k$ and if $(a_1,\ldots,a_n)$ is a sequence in $\mathbb Z_k$ such that $a_1+\cdots +a_n=0$, then there exists a {\it rearranged} subsequence $(x_1,\ldots,x_n)$ of $S$ such that $a_1x_1+\cdots+a_nx_n=0$. This extends the Erd\H{o}s-Ginzburg-Ziv theorem and confirms a conjecture of Y. Caro. When $n=|G|$, this result implies that $E_{A,\mathbf 1}(G)\leq 2n-1$ where $A=\{a_1,\ldots,a_n\}$.

Let $G$ be a finite abelian group and let $A\subseteq \mathbb Z$ be non-empty. In \cite{GMO}, it is shown that $E_A(G)=|G|+D_A(G)-1$ confirming a conjecture of Thangadurai and the expectations of Adhikari, et al.

\section{$(\mathbf 1,\mathbf 1)$-weighted zero-sum constants}

\begin{observation}\label{star}
Let $S$ be a sequence in $\mathbb Z_n$. If $S$ is a zero-sum sequence whose length is a multiple of $n$, we see that $S$ is a $(\mathbf 1,\mathbf 1)$-weighted zero-sum sequence. Conversely, if $S$ is a $(\mathbf 1,\mathbf 1)$-weighted zero-sum sequence, we see that $S$ is a zero-sum sequence whose length is a multiple of $n$.
\end{observation}

\begin{theorem}\label{e'1}
We have $D_{\mathbf 1,\mathbf 1}(n)=E_{\mathbf 1,\mathbf 1}(n)=2n-1$.
\end{theorem}

\begin{proof}
From \eqref{ubzn} we see that $D_{\mathbf 1,\mathbf 1}(n)\leq E_{\mathbf 1,\mathbf 1}(n)\leq 2n-1$. Consider the sequence
\[S=\big(\underbrace{\,0,\ldots,0}_{n-1},\,\underbrace{1,\ldots,1}_{n-1}\big).\]
We note that $S$ does not have any zero-sum subsequence whose length is a multiple of $n$. Thus, by Observation \ref{star} it follows that $D_{\mathbf 1,\mathbf 1}(n)\geq 2n-1$. Hence, we conclude that $D_{\mathbf 1,\mathbf 1}(n)=E_{\mathbf 1,\mathbf 1}(n)=2n-1$.
\end{proof}

\begin{theorem}\label{11}
Let $S$ be a sequence in $\mathbb Z_n$. The following are equivalent.
\begin{enumerate}[label=(\arabic*),ref=(\arabic*)]
\item[$(a)$]
$S$ is a $D$-extremal sequence for $(\mathbf 1,\mathbf 1)$

\item[$(b)$]
$S$ is an $E$-extremal sequence for $(\mathbf 1,\mathbf 1)$

\item[$(c)$]
$S$ is an $E$-extremal sequence for $\mathbf 1$

\item[$(d)$]
$S$ is a translate of a sequence which is equivalent to $\big(\underbrace{\,0,\ldots,0}_{n-1},\,\underbrace{1,\ldots,1}_{n-1}\big)$
\end{enumerate}
\end{theorem}

\begin{proof}
By Theorem \ref{e'1} we see that $D_{\mathbf 1,\mathbf 1}(n)=E_{\mathbf 1,\mathbf 1}(n)=2n-1$. So by using Observation \ref{star} we see that $(a)$ and $(b)$ are equivalent. From \cite{EGZ} we see that $E_\mathbf 1(n)=2n-1=E_{\mathbf 1,\mathbf 1}(n)$. Thus, by Observation \ref{star} it follows that $(b)$ and $(c)$ are equivalent. By Lemma 4 of \cite{BD} we see that $(c)$ and $(d)$ are equivalent.
\end{proof}

\begin{theorem}\label{c'1}
We have $C_{\mathbf 1,\mathbf 1}(n)=n^2$.
\end{theorem}

\begin{proof}
From \eqref{ubzn} we see that $C_{\mathbf 1,\mathbf 1}(n)\leq n^2$. Consider the sequence
\[S=(\,\underbrace{0,\ldots,0}_{n-1},1,\,\underbrace{0,\ldots,0}_{n-1},1,\ldots,1,\,\underbrace{0,\ldots,0}_{n-1}\,)\]
in $\mathbb Z_n$ having length $n^2-1$ (in which there are exactly $n-1$ ones). We observe that if $T$ is a subsequence of consecutive terms of $S$ and if $T$ has length $n$, then $T$ has exactly one non-zero term.

So we see that $S$ does not have any zero-sum subsequence of consecutive terms whose length is a multiple of $n$. Thus, from Observation \ref{star} it follows that $S$ does not have any $(\mathbf 1,\mathbf 1)$-weighted zero-sum subsequence of consecutive terms. Hence, we conclude that $C_{\mathbf 1,\mathbf 1}(n)=n^2$.
\end{proof}

Let $S'=(x_1,\ldots,x_{n-1})$ be sequence in $\mathbb Z_n$ which is a $C$-extremal sequence for $\mathbf 1$. (Such sequences have been characterized in Theorem 2 of \cite{SKS}.) We note that all the $x_i$'s are non-zero. Consider the sequence
\begin{equation}\label{cext}
S=(\,\underbrace{0,\ldots,0}_{n-1},x_1,\,\underbrace{0,\ldots,0}_{n-1},x_2,\ldots,x_{n-1},\,\underbrace{0,\ldots,0}_{n-1}\,).
\end{equation}
We observe that if $T$ is a subsequence of consecutive terms of $S$ of length $n$, then $T$ has exactly one non-zero term.

Since $S'$ does not have any zero-sum subsequence of consecutive terms, we see that $S$ does not have any zero-sum subsequence of consecutive terms whose length is a multiple of $n$. Hence, from Observation \ref{star} it follows that $S$ is a $C$-extremal sequence for $(\mathbf 1,\mathbf 1)$.

\begin{remark}\label{c11}
There are $C$-extremal sequences for $(\mathbf 1,\mathbf 1)$ in $\mathbb Z_n$ which are not of the form as in \eqref{cext}. For example, by using Observation \ref{star} we can check that the sequence $(1,0,1)$ in $\mathbb Z_2$ is a $C$-extremal sequence for $(\mathbf 1,\mathbf 1)$. Also, we see that the sequences $(0,1,0,0,2,2,0,0)$ and $(0,1,0,0,1,0,0,1)$ in $\mathbb Z_3$ are $C$-extremal sequences for $(\mathbf 1,\mathbf 1)$.
\end{remark}

\section{$(\mathbb Z_n',\mathbf 1)$-weighted zero-sum constants}\label{sz1}

\begin{observation}\label{2t}
Let $S$ be a sequence in $\mathbb Z_n$ which has a repeated term $x$. Then we see that $(x,x)$ is a $(\mathbb Z_n',\mathbf 1)$-weighted zero-sum subsequence of $S$.
\end{observation}

\begin{observation}\label{3d}
Let $S=(x,y,z)$ be a sequence in $\mathbb Z_n$  whose terms are pairwise distinct. Let $a=y-z$, $b=z-x$, and $c=x-y$. Since $ax+by+cz=0$ and $a+b+c=0$, it follows that $S$ is a $(\mathbb Z_n',\mathbf 1)$-weighted zero-sum sequence.
\end{observation}

\begin{theorem}\label{dnz}
We have $D_{\mathbb Z_n',\mathbf 1}=3$.
\end{theorem}

\begin{proof}
Let $S$ be a sequence in $\mathbb Z_n$ of length three. By Observations \ref{2t} and \ref{3d} we see that $S$ has a $(\mathbb Z_n',\mathbf 1)$-weighted zero-sum subsequence. Hence, it follows that $D_{\mathbb Z_n',\mathbf 1}\leq 3$.

Consider the sequence $S=(0,1)$ in $\mathbb Z_n$. The only $\mathbb Z_n'$-weighted zero-sum subsequence of $S$ is $T=(0)$. However, $T$ is not a $(\mathbb Z_n',\mathbf 1)$-weighted zero-sum sequence. Thus, it follows that $D_{\mathbb Z_n',\mathbf 1}=3$.
\end{proof}

\begin{theorem}\label{cnz}
We have $C_{\mathbb Z_n',\mathbf 1}=4$.
\end{theorem}

\begin{proof}
Consider the sequence $S=(0,1,0)$ in $\mathbb Z_n$. The only $\mathbb Z_n'$-weighted zero-sum subsequence of consecutive terms of $S$ is $T=(0)$. However, $T$ is not a $(\mathbb Z_n',\mathbf 1)$-weighted zero-sum sequence. Thus, it follows that $C_{\mathbb Z_n',\mathbf 1}\geq 4$.

Let $S$ be a sequence in $\mathbb Z_n$ of length four. We claim that $S$ has a $(\mathbb Z_n',\mathbf 1)$-weighted zero-sum subsequence having consecutive terms. It will then follow that $C_{\mathbb Z_n',\mathbf 1}=4$. By Observation \ref{2t}, we may assume that no two consecutive terms of $S$ are equal.

Let $S=(x,y,z,w)$. Suppose $x\neq z$. By Observation \ref{3d}, we see that $(x,y,z)$ is a $(\mathbb Z_n',\mathbf 1)$-weighted zero-sum sequence. So we may assume that $x=z$. By a similar argument we may also assume that $y=w$. Since $x+y-z-w=0$, we see that $S$ is a $(\mathbb Z_n',\mathbf 1)$-weighted zero-sum sequence. Hence, our claim follows.
\end{proof}

\begin{definition}
Suppose $T$ is a subsequence of a sequence $S=(x_1,\ldots,x_k)$ and $J=\{i:x_i$ is a term of $T\}$. Let $l=k-|J|$ and let $f:[1,l]\to [1,k]\setminus J$ be the unique increasing bijection. Then $S-T$ denotes the sequence $\big(x_{f(1)},\ldots,x_{f(l)}\big)$. For example, if $S=(1,2,1,4,2)$ and $T=(2,4)$, then $S-T=(1,1,2)$.
\end{definition}

\begin{definition}
Let $S$ and $T$ be sequences in $\mathbb Z_n$ of lengths $k$ and $l$ respectively. Then $S+T$ denotes the sequence in $\mathbb Z_n$ of length $k+l$ which is obtained by concatenating the sequences $S$ and $T$.
\end{definition}

\begin{lemma}\label{znzs}
Let $S$ be a sequence in $\mathbb Z_n$ where $n\geq 3$. Suppose at least two terms of $S$ are units. Then $S$ is a $\mathbb Z_n'$-weighted zero-sum sequence.
\end{lemma}

\begin{proof}
By permuting the terms of $S$, we may assume that $S=(x,y,x_3,\ldots,x_m)$ where $x$ and $y$ are units. Let $A=x\,\mathbb Z_n'$ and $B=y\,\mathbb Z_n'$. Since $x$ and $y$ are units, we see that $A=B=\mathbb Z_n'$. Since $n\geq 3$, there exist $u,v\in \mathbb Z_n'$ such that $u\neq v$.

Since $A+u=\mathbb Z_n \setminus \{u\}$ and $A+v=\mathbb Z_n \setminus \{v\}$, we see that $A+B=\mathbb Z_n$ and hence $x\,\mathbb Z_n'+y\,\mathbb Z_n'=\mathbb Z_n$. Thus, there exist $a,b\in \mathbb Z_n'$ such that $ax+by+x_3+\cdots +x_m=0$. It follows that $S$ is a $\mathbb Z_n'$-weighted zero-sum sequence.
\end{proof}

\begin{remark}\label{znzd}
Let $S$ be a sequence in $\mathbb Z_n$. Suppose every non-zero term of $S$ is a zero-divisor. Then we see that $S$ is a $\mathbb Z_n'$-weighted zero-sum sequence.
\end{remark}

\begin{lemma}\label{z'}
Let $S$ be a sequence in $\mathbb Z_n$ where $n\geq 3$. Suppose $S$ has at least two non-zero terms. Then $S$ is a $\mathbb Z_n'$-weighted zero-sum sequence.
\end{lemma}

\begin{proof}
If no term of $S$ is a unit, then we are done by Remark \ref{znzd}. If $S$ has at least two units, we are done by Lemma \ref{znzs}. So we may assume that exactly one term of $S$ is a unit.

Thus, there exist $x,y\in \mathbb Z_n'$ such that $(x,y)$ is a subsequence of $S$ and no term of $S-(x,y)$ is a unit. By Remark \ref{znzd} we see that $S-(x,y)$ is a $\mathbb Z_n'$-weighted zero-sum sequence. Also, $(x,y)$ is a $\mathbb Z_n'$-weighted zero-sum sequence since $yx-xy=0$. Hence, we are done.
\end{proof}

By considering the sequence $(1,1,1,0)$ in $\mathbb Z_2$, we see that the statements of Lemmas \ref{znzs} and \ref{z'} do not hold when $n=2$.

\begin{theorem}\label{enz}
We have $E_{\mathbb Z_n',\mathbf 1}=n+1$ when $n\neq 3$.
\end{theorem}

\begin{proof}
By Theorem 6.1 of \cite{ACFKP} we see that $E_{\mathbb Z_n'}=n+1$. Hence, by Observation \ref{cca} it follows that $E_{\mathbb Z_n',\mathbf 1}\geq n+1$.

Since every sequence of length three in $\mathbb Z_2$ has a term which is repeated, by Observation \ref{2t} it follows that $E_{\mathbb Z_2',\mathbf 1}=3$. Let $n\geq 4$ and let $S$ be a sequence in $\mathbb Z_n$ of length $n+1$. Since $S$ has length $n+1$, there is a term $x$ which is repeated. If $S-x$ has a $(\mathbb Z_n',\mathbf 1)$-weighted zero-sum subsequence of length $n$, by Observation \ref{tr}, we see that $S$ has a $(\mathbb Z_n',\mathbf 1)$-weighted zero-sum subsequence of length $n$. So we may assume that $S$ has at least two zeroes.

Let $T$ be a subsequence of $S$ such that $S-T=(0,0)$. Then $T$ has length $n-1\geq 3$. We claim that $T$ has a $\mathbb Z_n'$-weighted zero-sum subsequence $T'$ of length $n-2$. If $T$ has at most one non-zero term, our claim is true. Suppose $T$ has at least two non-zero terms. There exists a term $y$ of $T$ such that the sequence $T-(y)$ has at least two non-zero terms. Thus, by Lemma \ref{z'} we see that $T-(y)$ is a $\mathbb Z_n'$-weighted zero-sum sequence, and our claim is true.

We may assume that $T'=(x_3,\ldots,x_n)$ and that $x_1=0=x_2$. Observe that $(0,0)+T'$ has length $n$. There exist $a_3,\ldots,a_n\in \mathbb Z_n'$ such that $a_3x_3+\cdots+a_nx_n=0$. Let $c=a_3+\cdots+a_n$. Since $n\geq 4$, we see that $\mathbb Z_n'\setminus \{c\}\neq \emptyset$. Let $d\in \mathbb Z_n'\setminus \{c\}$, let $a_1=-d$, and let $a_2=d-c$. Then $a_1+\cdots+a_n=0$ and $a_1x_1+\cdots +a_nx_n=0$ and hence $(0,0)+T'$ is a $(\mathbb Z_n',\mathbf 1)$-weighted zero-sum sequence.
\end{proof}

\begin{theorem}\label{enz3}
We have $E_{\mathbb Z_3',\mathbf 1}=5$.
\end{theorem}

\begin{proof}
Let $S=(x_1,\ldots,x_5)$ be a sequence in $\mathbb Z_3$. If we show that $S$ has a $(\mathbb Z_3',\mathbf 1)$-weighted zero-sum subsequence of length three, it will follow that $E_{\mathbb Z_3',\mathbf 1}\leq 5$.

If $S$ has at least three distinct terms, by Observation \ref{3d} we see that $S$ has a $(\mathbb Z_3',\mathbf 1)$-weighted zero-sum subsequence of length three. If $S$ has at most two distinct terms, there exists $x\in\mathbb Z_3$ such that $T=(x,x,x)$ is a subsequence of $S$. Then $T$ is a $(\mathbb Z_3',\mathbf 1)$-weighted zero-sum subsequence. Hence, it follows that $E_{\mathbb Z_3',\mathbf 1}\leq 5$.

Consider the sequence $S=(0,0,1,1)$ in $\mathbb Z_3$. We can check that $T=(0,1,1)$ is the only $\mathbb Z_3'$-weighted zero-sum subsequence of length three. Since $T$ is not a $(\mathbb Z_3',\mathbf 1)$-weighted zero-sum subsequence, it follows that $E_{\mathbb Z_3',\mathbf 1}=5$.
\end{proof}

\section{Extremal sequences for $(\mathbb Z_n',\mathbf 1)$}

\begin{remark}\label{zn2}
Let $x,y \in \mathbb Z_n$ be such that $y-x$ is not a unit. So there exists $a\in \mathbb Z_n'$ such that $a(y-x)=0$ and hence $(-a)x+ay=0$. It follows that $(x,y)$ is a $(\mathbb Z_n',\mathbf 1)$-weighted zero-sum sequence.
\end{remark}

\begin{theorem}\label{ednz}
A sequence $S$ in $\mathbb Z_n$ is a $D$-extremal sequence for $(\mathbb Z_n',\mathbf 1)$ if and only if $S$ is a translate of a sequence which is equivalent to $S'=(0,1)$.
\end{theorem}

\begin{proof}
Let $S$ be a $D$-extremal sequence for $(\mathbb Z_n',\mathbf 1)$. By Theorem \ref{dnz} we see that $D_{\mathbb Z_n',\mathbf 1}=3$. It follows that $S$ has length two. Let $S=(x,y)$ and let $u=y-x$. By Remark \ref{zn2} we see that $u$ is a unit. It follows that $S=(0,u)+x$ and that $(0,u)$ is equivalent to $S'=(0,1)$.

Let $S$ be a translate of a sequence which is equivalent to $S'=(0,1)$. The only $\mathbb Z_n'$-weighted zero-sum subsequence of $S'=(0,1)$ is $T=(0)$. Since $T$ is not a $(\mathbb Z_n',\mathbf 1)$-weighted zero-sum sequence and since $D_{\mathbb Z_n',\mathbf 1}=3$, it follows that $S'$ is a $D$-extremal sequence for $(\mathbb Z_n',\mathbf 1)$. Thus, by Observation \ref{tr} and Remark \ref{eq} we see that $S$ is a $D$-extremal sequence for $(\mathbb Z_n',\mathbf 1)$.
\end{proof}

\begin{theorem}\label{ecnz}
A sequence $S$ in $\mathbb Z_n$ is a $C$-extremal sequence for $(\mathbb Z_n',\mathbf 1)$ if and only if $S$ is a translate of a sequence which is order-equivalent to $S'=(0,1,0)$.
\end{theorem}

\begin{proof}
Let $S$ be a $C$-extremal sequence for $(\mathbb Z_n',\mathbf 1)$. By Theorem \ref{cnz} we see that $C_{\mathbb Z_n',\mathbf 1}=4$. It follows that $S$ has length three. Let $S=(x,y,z)$, let $u=y-x$, and let $v=y-z$. By Remark \ref{zn2} we see that $u$ and $v$ are units. By Observation \ref{3d} we see that $z=x$ and hence $u=v$.  It follows that $S=(0,u,0)+x$ and that $(0,u,0)$ is order-equivalent to $S'=(0,1,0)$.

Let $S$ be a translate of a sequence which is order-equivalent to $S'=(0,1,0)$. The only $\mathbb Z_n'$-weighted zero-sum subsequence of consecutive terms of $S'$ is $T=(0)$. Since $T$ is not a $(\mathbb Z_n',\mathbf 1)$-weighted zero-sum sequence and since $C_{\mathbb Z_n',\mathbf 1}=4$, it follows that $S'$ is a $C$-extremal sequence for $(\mathbb Z_n',\mathbf 1)$. Thus, by Observation \ref{tr} and Remark \ref{eq} we see that $S$ is a $C$-extremal sequence for $(\mathbb Z_n',\mathbf 1)$.
\end{proof}

\begin{theorem}\label{eenz}
Let $n\neq 3$. A sequence $S$ in $\mathbb Z_n$ is an $E$-extremal sequence for $(\mathbb Z_n',\mathbf 1)$ if and only if $S$ is a translate of an $E$-extremal sequence for $\mathbb Z_n'$.
\end{theorem}

\begin{proof}
Let $T$ be an $E$-extremal sequence for $\mathbb Z_n'$ and let $S$ be a translate of $T$. By Theorem 6.1 of \cite{ACFKP} we have $E_{\mathbb Z_n'}=n+1$, and hence we see that $T$ has length $n$. Thus, it follows that $T$ is not a $\mathbb Z_n'$-weighted zero-sum sequence, and hence $T$ is not a $(\mathbb Z_n',\mathbf 1)$-weighted zero-sum sequence. By Theorem \ref{enz} we have $E_{\mathbb Z_n',\mathbf 1}=n+1$. So we see that $T$ is an $E$-extremal sequence for $(\mathbb Z_n',\mathbf 1)$. Thus, by Observation \ref{tr} we see that $S$ is an $E$-extremal sequence for $(\mathbb Z_n',\mathbf 1)$.

Let $S$ be an $E$-extremal sequence for $(\mathbb Z_n',\mathbf 1)$. Since $E_{\mathbb Z_n',\mathbf 1}=n+1$, we see that $S$ has length $n$. When $n=2$, by Observation \ref{2t} we see that $S=(0,1)$ or $(1,0)$. We observe that these are $E$-extremal sequences for $\mathbb Z_2'$. So we may assume that $n\geq 4$. Suppose all the terms of $S$ are distinct. Then $S$ is a permutation of the sequence $(0,1,2,\ldots,n-1)$. Since $2\in \mathbb Z_n'$ and since we have
\[2\,(0+1+2+\cdots+n-1)=(n-1)n=0,\]
we see that $S$ is a $(\mathbb Z_n',\mathbf 1)$-weighted zero-sum sequence. This contradicts that $S$ is an $E$-extremal sequence for $(\mathbb Z_n',\mathbf 1)$. It follows that there exists $x\in \mathbb Z_n$ such that $(x,x)$ is a subsequence of $S$. Let $S_1=S-x$. By Observation \ref{tr} it follows that $S_1$ is an $E$-extremal sequence for $(\mathbb Z_n',\mathbf 1)$.

We see that $S_1$ is a permutation of $S_2=(0,0,x_3,\ldots,x_n)$. Suppose $S_2$ is a $\mathbb Z_n'$-weighted zero-sum sequence. Then $(x_3,\ldots,x_n)$ is also a $\mathbb Z_n'$-weighted zero-sum sequence. By a similar argument as in the last paragraph of the proof of Theorem \ref{enz}, we see that $S_2$ is a $(\mathbb Z_n',\mathbf 1)$-weighted zero-sum sequence. As $S_1$ is a permutation of $S_2$, we get the contradiction that $S_1$ is a $(\mathbb Z_n',\mathbf 1)$-weighted zero-sum sequence. This contradiction shows that $S_2$ is not a $\mathbb Z_n'$-weighted zero-sum sequence.

Since $E_{\mathbb Z_n'}=n+1$ and since $S_2$ has length $n$, we see that $S_2$ is an $E$-extremal sequence for $\mathbb Z_n'$. As $S_1$ is a permutation of $S_2$, we see that $S_1$ is an $E$-extremal sequence for $\mathbb Z_n'$. Since $S=S_1+x$, we are done.
\end{proof}

\begin{remark}
In Theorem 1 of \cite{AMP} it is shown that a sequence $S$ in $\mathbb Z_n$ is an $E$-extremal sequence for $\mathbb Z_n'$ if and only if $S$ is equivalent to the sequence
\[(\,\underbrace{0,\ldots,0}_\text{$n-1$},1).\]
\end{remark}

\begin{theorem}
A sequence $S$ in $\mathbb Z_3$ is an $E$-extremal sequence for $(\mathbb Z_3',\mathbf 1)$ if and only if $S$ is a translate of a sequence which is equivalent to $S'=(0,0,1,1)$.
\end{theorem}

\begin{proof}
Let $S$ be an $E$-extremal sequence for $(\mathbb Z_3',\mathbf 1)$. By Theorem \ref{enz3} we see that $E_{\mathbb Z_3',\mathbf 1}=5$. It follows that $S$ has length 4. By using Observation \ref{3d} we see that $S$ has at most two distinct terms. Thus, there exists $y\in \mathbb Z_3'$ such that $S_1=S-y$ has at least two zeroes.

By Observation \ref{tr} we see that $S_1$ is an $E$-extremal sequence for $(\mathbb Z_3',\mathbf 1)$. As the sequence $(0,0,0)$ is a $(\mathbb Z_3',\mathbf 1)$-weighted zero-sum sequence, it follows that $S_1$ has exactly two zeroes. Since $S_1$ has at most two distinct terms, we see that $S_1$ is a permutation of $S_2=(0,0,x,x)$ where $x\in\mathbb Z_3'$. Since $x$ is a unit, it follows that $S_1$ is equivalent to $S'=(0,0,1,1)$ and $S=S_1+y$.

Let $S$ be a translate of a sequence which is equivalent to $S'=(0,0,1,1)$. The only $\mathbb Z_3'$-weighted zero-sum subsequence of $S'$ of length three is $T=(0,1,1)$. Suppose $T$ is a $(\mathbb Z_3',\mathbf 1)$-weighted zero-sum sequence. Then there exist $a,b,c\in \mathbb Z_3'$ such that $b+c=0$ and $a+b+c=0$. This contradicts that $a\neq 0$.

It follows that $S'$ has no $(\mathbb Z_3',\mathbf 1)$-weighted zero-sum subsequence of length three. Since $E_{\mathbb Z_3',\mathbf 1}=5$, it follows that $S'$ is an $E$-extremal sequence for $(\mathbb Z_3',\mathbf 1)$. Thus, by Observation \ref{tr} and Remark \ref{eq} we see that $S$ is an $E$-extremal sequence for $(\mathbb Z_3',\mathbf 1)$.
\end{proof}

\section{$(A,\mathbb Z_n')$-weighted zero-sum constants}

Let $x\in \mathbb Z_n\setminus \{0\}$. We note that $(x)$ is a $\mathbb Z_n'$-weighted zero-sum sequence if and only if $x$ is a zero-divisor.

\begin{observation}\label{1zn}
Let $A\subseteq \mathbb Z_n'$ and let $T=(x_1,\ldots,x_k)$ be a sequence in $\mathbb Z_n$ where $n\geq 3$. Suppose $T$ is an $A$-weighted zero-sum sequence of length $k\geq 2$. By Lemma \ref{z'} it follows that $T$ is an $(A,\mathbb Z_n')$-weighted zero-sum sequence.
\end{observation}

\begin{observation}\label{2zn}
Let $A\subseteq \mathbb Z_n'$ and let $T=(x)$ be an $A$-weighted zero-sum sequence in $\mathbb Z_n$ where $x\neq 0$. Then $T$ is an $(A,\mathbb Z_n')$-weighted zero-sum sequence, since there exists $a\in A$ such that $ax=0$.

Let $A\nsubseteq U(n)$. Then $T=(0)$ is an $(A,\mathbb Z_n')$-weighted zero-sum sequence, since there exists $a\in A$ and there exists $b\in \mathbb Z_n'$ such that $ba=0$ and $a0=0$.
\end{observation}

\begin{theorem}\label{daan}
Let $A\subseteq \mathbb Z_n'$. We have $D_A\leq D_{A,\mathbb Z_n'}\leq D_A+1$.
\end{theorem}

\begin{proof}
By Observation \ref{cca} we see that $D_A\leq D_{A,\mathbb Z_n'}$. By Theorem \ref{e'1} we see that $D_{\mathbb Z_2',\mathbb Z_2'}=3$. Also, since $D_{\mathbb Z_2'}=2$, it follows that $D_{\mathbb Z_2',\mathbb Z_2'}=D_{\mathbb Z_2'}+1$. So we may assume that $n\geq 3$.

Let $S$ be a sequence in $\mathbb Z_n$ of length $D_A+1$. Since $S$ has length at least $D_A$, we see that $S$ has an $A$-weighted zero-sum subsequence $T_1$. If $T_1$ has length one, then $S-T_1$ has an $A$-weighted zero-sum subsequence $T_2$, since $S-T_1$ has length $D_A$. Hence, $T_1+T_2$ is an $A$-weighted zero-sum subsequence of $S$ of length at least two.

Thus, we see that $S$ has an $A$-weighted zero-sum subsequence $T$ of length at least two.
By Observation \ref{1zn} we see that $T$ is an $(A,\mathbb Z_n')$-weighted zero-sum sequence. Hence, it follows that $D_{A,\mathbb Z_n'}\leq D_A+1$.
\end{proof}

\begin{theorem}\label{dan}
Let $A\subseteq \mathbb Z_n'$. We have
\[D_{A,\mathbb Z_n'}=
\begin{cases}
D_A+1, & \emph {if }A\subseteq U(n); \\
D_A, & \emph{if }A\nsubseteq U(n).
\end{cases}\]
\end{theorem}

\begin{proof}
The case when $n=2$ follows from the first paragraph of the proof of Theorem \ref{daan}. So we may assume that $n\geq 3$.

Let $A\subseteq U(n)$. There exists a sequence $S'$ in $\mathbb Z_n$ of length $D_A-1$ such that $S'$ has no $A$-weighted zero-sum subsequence. Let $S=S'+(0)$. The only $A$-weighted zero-sum subsequence of $S$ is $T=(0)$.

Since $A\subseteq U(n)$, we see that $T$ is not an $(A,\mathbb Z_n')$-weighted zero-sum sequence. Thus, we see that $S$ does not have any $(A,\mathbb Z_n')$-weighted zero-sum subsequence. It follows that $D_{A,\mathbb Z_n'}\geq D_A+1$. So by Theorem \ref{daan} we see that $D_{A,\mathbb Z_n'}=D_A+1$.

Let $A\nsubseteq U(n)$. Let $S$ be a sequence in $\mathbb Z_n$ of length $D_A$. Then $S$ has an $A$-weighted zero-sum subsequence $T$. By Observations \ref{1zn} and \ref{2zn} we see that $T$ is an $(A,\mathbb Z_n')$-weighted zero-sum sequence. Hence, it follows that $D_{A,\mathbb Z_n'}\leq D_A$. Thus, by Theorem \ref{daan} we see that $D_{A,\mathbb Z_n'}=D_A$.
\end{proof}

\begin{theorem}\label{caan}
Let $A\subseteq \mathbb Z_n'$. We have $C_A\leq C_{A,\mathbb Z_n'}\leq 2\,C_A$.
\end{theorem}

\begin{proof}
By Observation \ref{cca} we see that $C_A\leq C_{A,\mathbb Z_n'}$. By Theorem \ref{c'1} and Corollary 1 of \cite{SKS1} we see that $C_{\mathbb Z_2',\mathbb Z_2'}=4$ and $C_{\mathbb Z_2'}=2$. Hence, it follows that $C_{\mathbb Z_2',\mathbb Z_2'}=2\,C_{\mathbb Z_2'}$. So we may assume $n\geq 3$. Let
\[E=\Big\{x\in \mathbb Z_n:\text{there exists }a\in A\text{ such that }ax=0\Big\}.\]

Let $S$ be a sequence in $\mathbb Z_n$ having length $2\,C_A$. Suppose there exist $x,y\in E$ such that $x$ and $y$ are consecutive terms of $S$. Let $T=(x,y)$. Then $T$ is an $A$-weighted zero-sum subsequence of consecutive terms. By Observation \ref{1zn} we see that $T$ is an $(A,\mathbb Z_n')$-weighted zero-sum sequence.

Suppose no two consecutive terms of $S$ are in $E$. Let $S'$ be the subsequence consisting of all the terms of $S$ which are not in $E$. Since $S$ has length $2\,C_A$, we see that the length of $S'$ is at least $C_A$. It follows that $S'$ has an $A$-weighted zero-sum subsequence $T'$ of consecutive terms.

If $T'=(x)$, we see that $x\in E$ since there exists $a\in A$ such that $ax=0$. This contradicts the fact that no term of $S'$ is in $E$. Thus, we see that $T'$ has length at least two. Let $T$ be the subsequence of $S$ whose first and last terms are the first and last terms of $T'$ respectively.

If $x$ is a term of $S-S'$, then $x\in E$, and hence $(x)$ is an $A$-weighted zero-sum sequence. So we see that $T$ is an $A$-weighted zero-sum subsequence of consecutive terms of $S$ of length at least three.

Thus, in both the cases, we see that $S$ has an $A$-weighted zero-sum subsequence of consecutive terms of length at least two. By Observation \ref{1zn} we see that $T$ is an $(A,\mathbb Z_n')$-weighted zero-sum sequence. Hence, it follows that $C_{A,\mathbb Z_n'}\leq 2\,C_A$.
\end{proof}

\begin{theorem}\label{can}
Let $A\subseteq \mathbb Z_n'$. We have
\[C_{A,\mathbb Z_n'}=
\begin{cases}
2\,C_A, & \emph {if }A\subseteq U(n); \\
C_A, & \emph{if }A\nsubseteq U(n).
\end{cases}\]
\end{theorem}

\begin{proof}
The case $n=2$ follows from the first paragraph of the proof of Theorem \ref{caan}. So we may assume that $n\geq 3$.

Let $A\subseteq U(n)$. There exists a sequence $S'=(x_1,\ldots,x_k)$ in $\mathbb Z_n$ of length $C_A-1$ which does not have any $A$-weighted zero-sum subsequence of consecutive terms. Let $S=(0,x_1,0,x_2,0,\ldots,x_k,0)$. Then $S$ has length $2\,C_A-1$. The only $A$-weighted zero-sum subsequence of consecutive terms of $S$ is $T=(0)$.

Since $A\subseteq U(n)$, we see that $T=(0)$ is not an $(A,\mathbb Z_n')$-weighted zero-sum sequence. Thus, we see that $S$ does not have any $(A,\mathbb Z_n')$-weighted zero-sum subsequence of consecutive terms. So it follows that $C_{A,\mathbb Z_n'}\geq 2\,C_A$. Hence, by Theorem \ref{caan} we see that $C_{A,\mathbb Z_n'}=2\,C_A$.

Suppose $A\nsubseteq U(n)$. Let $S$ be a sequence in $\mathbb Z_n$ of length $C_A$. Then $S$ has an $A$-weighted zero-sum subsequence $T$ of consecutive terms. By Observations \ref{1zn} and \ref{2zn} we see that $T$ is an $(A,\mathbb Z_n')$-weighted zero-sum sequence. Hence, it follows that $C_{A,\mathbb Z_n'}\leq C_A$. Thus, by Theorem \ref{caan} we see that $C_{A,\mathbb Z_n'}=C_A$.
\end{proof}

\begin{theorem}\label{ean}
Let $A\subseteq \mathbb Z_n'$. We have $E_{A,\mathbb Z_n'}=E_A$.
\end{theorem}

\begin{proof}
By Theorem \ref{e'1} we see that $E_{\mathbf 1,\mathbf 1}=3$ and from \cite{EGZ} we see that $E_{\mathbf 1}=3$. Hence, we are done when $n=2$. Let $n\geq 3$ and let $S$ be a sequence in $\mathbb Z_n$ of length $E_A$. Then $S$ has an $A$-weighted zero-sum subsequence $T$ of length $n$. By Observation \ref{1zn} we see that $T$ is an $(A,\mathbb Z_n')$-weighted zero-sum sequence. Thus, we see that $E_{A,\mathbb Z_n'}\leq E_A$. Hence, by Observation \ref{cca} it follows that $E_{A,\mathbb Z_n'}=E_A$.
\end{proof}

\section{Extremal sequences for $(A,\mathbb Z_n')$}

In this section, we characterize the extremal sequences for $(A,\mathbb Z_n')$ using the corresponding extremal sequences for $A$ where $A$ is a non-empty subset of $\mathbb Z_n'$.

\begin{remark}\label{zun}
Let $A\subseteq U(n)$ and let $x\in \mathbb Z_n$. If $(x)$ is an $A$-weighted zero-sum sequence, we see that $x=0$.
\end{remark}

\begin{theorem}\label{da'}
Let $A\subseteq U(n)$. Then a sequence $S$ is a $D$-extremal sequence for $(A,\mathbb Z_n')$ if and only if $S$ has a zero and $S-(0)$ is a $D$-extremal sequence for $A$.
\end{theorem}

\begin{proof}
The case $n=2$ follows from Theorem \ref{11}. So we may assume that $n\geq 3$.
Let $S$ be a $D$-extremal sequence for $(A,\mathbb Z_n')$. Since $A\subseteq U(n)$, by Theorem \ref{dan} we see that $D_{A,\mathbb Z_n'}=D_A+1$ and hence $S$ has length $D_A$. By Observation \ref{1zn} we see that $S$ cannot have any $A$-weighted zero-sum subsequence of length at least two. It follows that $S$ has at most one zero.

Let $S'$ be the subsequence consisting of all the non-zero terms of $S$. Suppose $T$ is an $A$-weighted zero-sum subsequence of $S'$. Since $T$ is a subsequence of $S$, we see that $T$ has length one. Since $A\subseteq U(n)$, by Remark \ref{zun} we see that $T=(0)$. This gives the contradiction that $S'$ has a zero.

Thus, we see that $S'$ does not have any $A$-weighted zero-sum subsequence. Hence, $S'$ has length at most $D_A-1$. Since $S$ has at most one zero, we see that $S'$ has length at least $D_A-1$. Since $S'$ has length $D_A-1$, it follows that $S'$ is a $D$-extremal sequence for $A$. Also, we see that $S$ must have a zero.

Let $0$ be a term of $S$ and let $S-(0)$ be a $D$-extremal sequence for $A$. Suppose $S$ has an $(A,\mathbb Z_n')$-weighted zero-sum subsequence $T$. Then $T$ is an $A$-weighted zero-sum subsequence of $S$. If $T\neq (0)$, then we get the contradiction that $S-(0)$ has an $A$-weighted zero-sum subsequence, and hence $T=(0)$.

Since $T$ is an $(A,\mathbb Z_n')$-weighted zero-sum sequence, there exists $a\in A$ and $b\in \mathbb Z_n'$ such that $b\;a=0$. Since $A\subseteq U(n)$, we get the contradiction that $b=0$. Thus, we see that $S$ does not have any $(A,\mathbb Z_n')$-weighted zero-sum subsequence. Since $S$ has length $D_A=D_{A,\mathbb Z_n'}-1$, it follows that $S$ is a $D$-extremal sequence for $(A,\mathbb Z_n')$.
\end{proof}

\begin{theorem}\label{dad}
Let $A\nsubseteq U(n)$. Then a sequence $S$ is a $D$-extremal sequence for $(A,\mathbb Z_n')$ if and only if $S$ is a $D$-extremal sequence for $A$.
\end{theorem}

\begin{proof}
Since $A\nsubseteq U(n)$, we see that $n\geq 3$, and by Theorem \ref{dan} we see that $D_{A,\mathbb Z_n'}=D_A$. By Observations \ref{1zn} and \ref{2zn} we see that a sequence $T$ is an $(A,\mathbb Z_n')$-weighted zero-sum sequence if and only if $T$ is an $A$-weighted zero-sum sequence. Hence, it follows that a sequence $S$ is a $D$-extremal sequence for $A$ if and only if $S$ is a $D$-extremal sequence for $(A,\mathbb Z_n')$.
\end{proof}

By using a similar argument along with Theorem \ref{can}, we get the next result.

\begin{theorem}\label{cad}
Let $A\nsubseteq U(n)$. Then a sequence $S$ is a $C$-extremal sequence for $(A,\mathbb Z_n')$ if and only if $S$ is a $C$-extremal sequence for $A$.
\end{theorem}

From Remark \ref{c11} we see that the next result does not hold when $n=2$.

\begin{theorem}\label{ca'}
Let $A\subseteq U(n)$ where $n\geq 3$. Then a sequence $S$ is a $C$-extremal sequence for $(A,\mathbb Z_n')$ if and only if there exists a sequence $S'=(x_1,\ldots,x_k)$ which is a $C$-extremal sequence for $A$ such that
\begin{equation}\label{cex}
S=(0,x_1,0,x_2,0,\ldots,x_k,0).
\end{equation}
\end{theorem}

\begin{proof}
Let $S$ be a $C$-extremal sequence for $(A,\mathbb Z_n')$. By Observation \ref{1zn} we see that $S$ cannot have an $A$-weighted zero-sum subsequence of consecutive terms of length at least two. Since $A\subseteq U(n)$, by Theorem \ref{can} we see that $C_{A,\mathbb Z_n'}=2\,C_A$ and hence $S$ has length $2\,C_A-1$. If $S$ has at least $C_A+1$ zeroes, then we get the contradiction that $S$ has two consecutive zeroes.

Let $S'$ be the subsequence consisting of all the non-zero terms of $S$. Since $S$ has length $2\,C_A-1$ and since $S$ has at most $C_A$ zeroes, we see that $S'$ has length at least $C_A-1$. Suppose $T$ is an $A$-weighted zero-sum subsequence of consecutive terms of $S'$. Then we see that $T$ has length one. By Remark \ref{zun} we get the contradiction that $S'$ has a zero.

Thus, we see that $S'$ does not have any $A$-weighted zero-sum subsequence of consecutive terms. Since $S'$ has length at least $C_A-1$, it follows that $S'$ has length $C_A-1$. So we see that $S'$ is a $C$-extremal sequence for $A$. Also, since $S$ does not have any consecutive zeroes, we see that $S$ has the form as in \eqref{cex}.

Let $S$ and $S'$ be sequences as in \eqref{cex}. Suppose $S$ has an $(A,\mathbb Z_n')$-weighted zero-sum subsequence of consecutive terms $T$. Then $T$ is an $A$-weighted zero-sum subsequence of consecutive terms of $S$. Since $S'$ does not have any $A$-weighted zero-sum subsequence of consecutive terms, we see that $T=(0)$.

Since $T$ is an $(A,\mathbb Z_n')$-weighted zero-sum sequence, there exists $a\in A$ and $b\in \mathbb Z_n'$ such that $b\,a=0$. Since $A\subseteq U(n)$, we get the contradiction that $b=0$. Thus, we see that $S$ does not have any $(A,\mathbb Z_n')$-weighted zero-sum subsequence of consecutive terms. Since $S$ has length $2\,C_A-1=C_{A,\mathbb Z_n'}-1$, it follows that $S$ is a $C$-extremal sequence for $(A,\mathbb Z_n')$.
\end{proof}

\begin{theorem}\label{en'}
Let $A\subseteq \mathbb Z_n'$. Then $S$ is an $E$-extremal sequence for $(A,\mathbb Z_n')$ if and only if $S$ is an $E$-extremal sequence for $A$.
\end{theorem}

\begin{proof}
When $n=2$, we are done by Theorem \ref{11}. Let $n\geq 3$. By Observation \ref{1zn} we see that a sequence $T$ of length $n$ is an $(A,\mathbb Z_n')$-weighted zero-sum sequence if and only if $T$ is an $A$-weighted zero-sum sequence. By Theorem \ref{ean} we see that $E_{A,\mathbb Z_n'}=E_A$. Hence, it follows that a sequence $S$ is an $E$-extremal sequence for $A$ if and only if $S$ is an $E$-extremal sequence for $(A,\mathbb Z_n')$.
\end{proof}

\begin{remark}\label{comp}
Let $n\geq 4$. By Theorem 6.1 of \cite{ACFKP} we see that $E_{\mathbb Z_n'}=n+1$. So from Theorems \ref{enz} and \ref{ean} we see that $E_{\mathbb Z_n',\mathbf 1}=E_{\mathbb Z_n',\mathbb Z_n'}$. Since a $(\mathbb Z_n',\mathbf 1)$-weighted zero-sum sequence is also a $(\mathbb Z_n',\mathbb Z_n')$-weighted zero-sum sequence, it follows that an $E$-extremal sequence for $(\mathbb Z_n',\mathbb Z_n')$ is also an $E$-extremal sequence for $(\mathbb Z_n',\mathbf 1)$.
\end{remark}

\begin{theorem}
Let $n\geq 4$. Then an $E$-extremal sequence for $(\mathbb Z_n',\mathbf 1)$ is not an $E$-extremal sequence for $(\mathbb Z_n',\mathbb Z_n')$ if and only if it is a non-zero translate of a sequence which is equivalent to the sequence $S=\big(\underbrace{0,\ldots,0}_\text{$n-1$},1\big)$.
\end{theorem}

\begin{proof}
By Theorem \ref{en'} and by Theorem 1 of \cite{AMP} we see that the $E$-extremal sequences for $(\mathbb Z_n',\mathbb Z_n')$ are exactly the sequences which are equivalent to $S$. Hence, by Theorem \ref{eenz} and by Theorem 1 of \cite{AMP} we are done.
\end{proof}

When $n=2$, we see that $(\mathbb Z_n',\mathbb Z_n')=(\mathbb Z_n',\mathbf 1)$. By Theorems \ref{enz3} and \ref{ean} we see that $E_{\mathbb Z_3',\mathbb Z_3'}\neq E_{\mathbb Z_3',\mathbf 1}$. Hence, we cannot compare the $E$-extremal sequences for the constants $E_{\mathbb Z_3',\mathbb Z_3'}$ and $E_{\mathbb Z_3',\mathbf 1}$.

\section{$(\mathbb Z_n',B)$-weighted zero-sum constants}\label{b}

\begin{remark}\label{ub}
Let $B\subseteq\mathbb Z_n'$. If a sequence $S$ is a $(\mathbb Z_n',\mathbf 1)$-weighted zero-sum sequence, then $S$ is also a $(\mathbb Z_n',B)$-weighted zero-sum sequence.
\end{remark}

\begin{theorem}\label{bg}
Let $B\subseteq \mathbb Z_n'$. Then we have the following results:
\begin{enumerate}[label=(\arabic*),ref=(\arabic*)]
\item[$(a)$] $2\leq C_{\mathbb Z_n',B}\leq 4$.
\item[$(b)$] $2\leq D_{\mathbb Z_n',B}\leq 3$.
\item[$(c)$] $E_{\mathbb Z_n',B}=n+1\mbox{ when }n\neq 3\mbox{ or when }B=\mathbb Z_n'$.
\item[$(d)$] $E_{\mathbb Z_3'\,,\mathbf 1}=E_{\mathbb Z_3',\{-1\}}=5$.
\end{enumerate}
\end{theorem}

\begin{proof}
By Theorem 2 of \cite{SKS1} we see that $C_{\mathbb Z_n'}=D_{\mathbb Z_n'}=2$. By Observation \ref{cca} we see that $C_{\mathbb Z_n',B}\geq 2$ and $D_{\mathbb Z_n',B}\geq 2$. From Remark \ref{ub} it follows that $C_{\mathbb Z_n',B}\leq C_{\mathbb Z_n',\mathbf 1}$ and $D_{\mathbb Z_n',B}\leq D_{\mathbb Z_n',\mathbf 1}$. By Theorems \ref{dnz} and \ref{cnz} we see that $D_{\mathbb Z_n',\mathbf 1}=3$ and $C_{\mathbb Z_n',\mathbf 1}=4$. Thus, we get $(a)$ and $(b)$.

By Theorem 6.1 of \cite{ACFKP} we see that $E_{\mathbb Z_n'}=n+1$. So by Theorem \ref{ean} we see that $E_{\mathbb Z_n',\mathbb Z_n'}=n+1$ and by Observation \ref{cca} we see that $E_{\mathbb Z_n',B}\geq n+1$. By Remark \ref{ub} we see that $E_{\mathbb Z_n',B}\leq E_{\mathbb Z_n',\mathbf 1}$. When $n\neq 3$, by Theorem \ref{enz} we see that $E_{\mathbb Z_n',\mathbf 1}=n+1$. Thus, we get $(c)$. By Theorem \ref{enz3} we see that $E_{\mathbb Z_3',\mathbf 1}=5$. It follows that $E_{\mathbb Z_3'\,,\,\{-1\}}=5$. Thus, we get $(d)$.
\end{proof}

\begin{theorem}\label{dnb}
Let $B\subseteq U(n)$. We have $C_{\mathbb Z_n',B}=4$ and $D_{\mathbb Z_n',B}=3$.
\end{theorem}

\begin{proof}
Let $S_1=(0,1,0)$ and $S_2=(0,1)$. The only $\mathbb Z_n'$-weighted zero-sum subsequence of consecutive terms of $S_1$ is $T=(0)$. Since $B\subseteq U(n)$, we see that $T$ is not a $(\mathbb Z_n',B)$-weighted zero-sum sequence. Hence, by Theorem \ref{bg} it follows that $C_{\mathbb Z_n',B}=4$. By a similar argument, we see that $S_2$ does not have any $(\mathbb Z_n',B)$-weighted zero-sum subsequence. Hence, by Theorem \ref{bg} it follows that $D_{\mathbb Z_n',B}=3$.
\end{proof}

\section{Concluding remarks}

Let $R$ be a ring with unity, let $M$ be an $R$-module, and let $A$, $B$, $C$ be non-empty subsets of $R$. A sequence $(x_1,\ldots,x_k)$ in $M$ is called an {\it $(A,B,C)$-weighted zero-sum sequence} if there exist $a_1,\ldots,a_k\in A$, $b_1,\ldots,b_k\in B$, $c_1,\ldots,c_k\in C$ such that
\[a_1x_1+\cdots+a_kx_k=0,~b_1a_1+\cdots+b_ka_k=0,~c_1b_1+\cdots+c_kb_k=0.\]
We can define the $(A,B,C)$-weighted constants $D_{A,B,C}(M)$ and $E_{A,B,C}(M)$ in an analogous manner as in this article.

Let $B\subseteq U(n)$. From Theorems \ref{dnz} and \ref{dnb} we see that $D_{\mathbb Z_n',B}=D_{\mathbb Z_n',\mathbf 1}=3$. By Remark \ref{ub} it follows that a $D$-extremal sequence for $(\mathbb Z_n',B)$ is also a $D$-extremal sequence for $(\mathbb Z_n',\mathbf 1)$. The sequence $S=(0,1)$ is a $D$-extremal sequence for $(\mathbb Z_n',B)$. By Remark \ref{eq} and Theorem \ref{ednz} it remains to determine which translates of $S$ are $D$-extremal sequences for $(\mathbb Z_n',B)$.

Let $B\subseteq \mathbb Z_n'$. From Theorems \ref{enz}, \ref{enz3}, and \ref{bg} we see that $E_{\mathbb Z_n',B}=E_{\mathbb Z_n',\mathbf 1}$. So we can determine which $E$-extremal sequences for $(\mathbb Z_n',\mathbf 1)$ are also $E$-extremal sequences for $(\mathbb Z_n',B)$. The constants $C_{\mathbb Z_n',B}$ and $D_{\mathbb Z_n',B}$ have been computed when $B\subseteq U(n)$ and when $B=\mathbb Z_n'$. Their values may be found for other subsets $B\subseteq \mathbb Z_n'$.

\end{document}